\documentclass[letterpaper, 12pt]{article}
\usepackage{amsmath,amsthm,amssymb,enumerate}
\usepackage{tikz, subfig}
\usepackage{fullpage}
\usepackage[normalem]{ulem}

\def\new{\color[rgb]{0.60,0.15,0}}

\newcount\sracnum

\def\comment#1{{
\global\advance \sracnum by 1 ${}^{[\the\sracnum]}$
\marginpar{
\vskip-4mm{\tiny~~(\the\sracnum)}~~
 {\footnotesize \baselineskip=10pt \raggedright #1\\ ~\\~}}}}

\newtheorem{thm}{Theorem}

\newtheorem{conjecture}[thm]{Conjecture}

\newtheorem{prop}[thm]{Proposition}

\newtheorem{lem}[thm]{Lemma}

\theoremstyle{definition}

\newtheorem*{defin*}{Definition}

\newtheorem{case}{Case}

\def\t{t} 

\def\text#1{\quad\textnormal{#1}\quad}

\def\Text#1{\qquad\textnormal{#1}\qquad}

\def\tpt{{tripartite}}
\def\EE{\mathbb{E}}
\def\M{\mathcal{M}}
\def\N{\mathbb{N}}
\def\K{\mathcal{K}}

\def\HH{\mathcal{H}}

\def\alp{{\alpha}}

\def\B{\mathcal{B}}

\def\E{\mathcal{E}}
\def\FF{\mathcal{F}}
\def\G{\mathcal{G}}

\def\F{\hat{F}}

\def\P{\mathbb{P}}

\def\eps{\varepsilon}
\def\ph{\varphi}

\def\G{\mathcal{G}}

\def\sg{{J}}

\def\beq#1{\begin{equation}\label{#1}}
\def\eeq{\end{equation}}

\def\bth{\begin{thm}}

\def\eth{\end{thm}}

\def\bc{\begin{corollary}}

\def\ec{\end{corollary}}

\def\bcj{\begin{conjecture}}

\def\ecj{\end{conjecture}}

\def\whp{w.h.p.\@ }

\title{Mantel's Theorem for Random Hypergraphs}
\author{
 J\'{o}zsef Balogh \footnote{
    Department of Mathematics, University of Illinois, Urbana, IL 61801, USA, {\tt jobal@math.uiuc.edu}.
    Research is partially supported by Simons Fellowship, NSF CAREER Grant DMS-0745185 and Arnold O. Beckman Research Award (UIUC Campus Research Board 13039).}
 \and Jane Butterfield \footnote{Department of Mathematics \& Statistics, University Of Victoria, Victoria BC V8P 5C2, Canada {\tt jvbutter@uvic.ca}. Research is partially supported by NSF Grant DMS 08-38434 "EMSW21-MCTP: Research Experience for Graduate Students''.}
 \and Ping Hu \footnote{Department of Mathematics, University of Illinois, Urbana, IL 61801, USA, {\tt pinghu1@math.uiuc.edu}.}
 \and John Lenz \footnote{Department of Mathematics, Statistics, and Computer Science,
University of Illinois at Chicago, 851
S. Morgan Street, Chicago, IL, USA, {\tt lenz@math.uic.edu}. Research is partially supported by the National Security Agency Grant
H98230-13-1-0224.}
}
\date{}

\begin{document}
\maketitle

\begin{abstract}
A classical result in extremal graph theory is Mantel's Theorem, which states that every 
maximum triangle-free subgraph of $K_n$ is bipartite. A sparse version of Mantel's Theorem 
 is that, for sufficiently large $p$, every maximum triangle-free subgraph of $G(n,p)$ is \whp bipartite. Recently, DeMarco and Kahn proved this for $p > K \sqrt{\log n/n}$ for some constant $K$, and  apart from the value of the constant this bound is best possible. 
 
 We study an extremal problem of this type in random hypergraphs. Denote by $F_5$, which sometimes called as the generalized triangle,  the $3$-uniform hypergraph with vertex set $\left\{ a,b,c,d,e \right\}$ and edge set $\{abc, ade, bde\}$. One of the first  extremal results in extremal hypergraph theory is by Frankl and F\"{u}redi, who  proved that  the maximum $3$-uniform hypergraph on $n$
vertices containing no copy of $F_5$ is {\tpt} for $n>3000$.
A natural question is for what $p$ is every maximum $F_5$-free subhypergraph of $G^3(n,p)$  w.h.p.\@ {\tpt}.
 We show this holds for $p>K\log n/n$ for some constant $K$ and does not hold for $p=0.1\sqrt{\log n}/n$.
 
 \end{abstract}

\section{Introduction}

A classical result in extremal graph theory is Mantel's Theorem~\cite{mantel}, which states that every $K_3$-free graph on $n$ vertices has at most $\lfloor n^2/4 \rfloor$ edges. 
Furthermore, the complete bipartite graph whose partite sets differ in size by at most 
one is the unique $K_3$-free graph that achieves this bound. In other words, every maximum 
(with respect to the number of edges) triangle-free subgraph of $K_n$ is bipartite. 

A sparse version of Mantel's Theorem has recently been proved by DeMarco and Kahn~\cite{DK}:
 Let {$G(n,p)$}  be the usual Erd\H{o}s-R\'{e}nyi random graph.
An event occurs 
\emph{with high probability} (w.h.p.) if the probability of that event approaches $1$ as $n$ tends to infinity.
We are interested to determine for what $p$ every maximum triangle-free subgraph of $G(n,p)$ is w.h.p.\@ bipartite. 
DeMarco and Kahn proved that  every maximum triangle-free subgraph of $G(n,p)$ is w.h.p.\@ bipartite   if $p > K \sqrt{\log n/n}$ for some large constant $K$.   If $p=0.1\sqrt{\log n/n}$, then \whp there is   a $C_5$  in $G(n,p)$ whose edges are not in any triangle,  therefore any maximum triangle-free subgraph of $G(n,p)$ contains this $C_5$ and is not bipartite. So apart from the value of the constant the result of  DeMarco and Kahn is best possible.

Problems of this type were first considered by Babai, Simonovits and Spencer~\cite{MR1073101}. 
Bright\-well, Panagiotou and Steger~\cite{MR2482874} 
proved the existence of a constant $c$, depending only on $\ell$, such that whenever 
$p \geq n^{-c}$, w.h.p.\@ every maximum $K_{\ell}$-free subgraph of $G(n,p)$ is 
($\ell-1$)-partite, and recently, DeMarco and Kahn~\cite{DK2}  found the appropriate range of $p$ for this problem.
Here, we study an extremal problem of this type in random hypergraphs.  
\begin{defin*}
	For $n \in \mathbb{N}$ and $p \in [0,1]$, let $G^r(n,p)$ be a random $r$-uniform hypergraph with $n$ 
	vertices and  each element of ${[n] \choose r}$ 
	occurring as an edge with probability $p$  independently of each other.
	In particular, $G^2(n,p) = G(n,p)$. 
Denote by $F_5$ the $3$-uniform hypergraph with vertex set $\left\{ a,b,c,d,e \right\}$ and edge set $\{abc, ade, bde\}$. Denote by $K_4^-$ the $3$-uniform hypergraph with $4$ vertices and $3$ edges.
\end{defin*}

\vspace{-3ex}
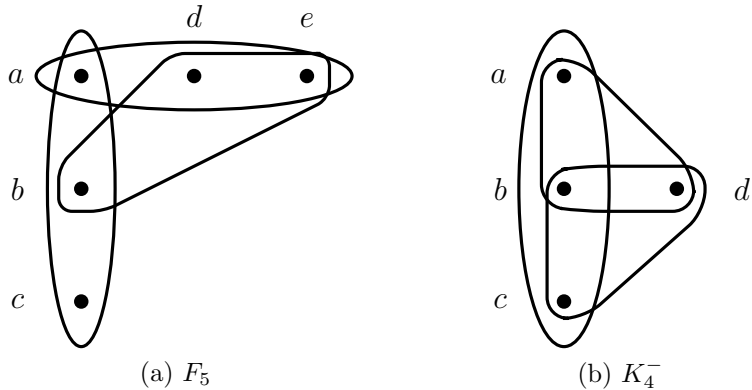
\begin{figure}[h]
\begin{center}
\subfloat[$F_5$]{
\begin{tikzpicture}[scale=1.5, very thick]

	\filldraw  (0,0) circle (.05cm);
	\node [left = 0.6cm] at (0,0) {$a$};	
	\filldraw (0, -1) circle (.05cm);
	\node [left = 0.6cm] at (0,-1) {$b$};
	\filldraw (0, -2) circle (.05cm);
	\node [left = 0.6cm] at (0,-2) {$c$};
	\filldraw (1, 0) circle (.05cm);
	\node [above = 0.5cm] at (1,0) {$d$};
	\filldraw (2, 0) circle (.05cm);
	\node [above = 0.5cm] at (2,0) {$e$};

	\draw (0,-1) ellipse (.3 and 1.4); 
	\draw (1,0) ellipse (1.4 and .3);
	\draw[rounded corners=5pt] (2,.2) -- (2.2, .2) -- (2.2, -.2) -- (.2, -1.2) -- (-.2, -1.2) -- (-.2, -.8) --(.8, .2) -- cycle;
\end{tikzpicture}
}
\hspace{0.5in}
\subfloat[$K_4^-$]{
\begin{tikzpicture}[scale=1.5, very thick]

	\filldraw  (0,0) circle (.05cm);
	\node [left = 0.6cm] at (0,0) {$a$};	
	\filldraw (0, -1) circle (.05cm);
	\node [left = 0.6cm] at (0,-1) {$b$};
	\filldraw (0, -2) circle (.05cm);
	\node [left = 0.6cm] at (0,-2) {$c$};
	\filldraw (1, -1) circle (.05cm);
	\node [right = 0.6cm] at (1,-1) {$d$};

	\draw (0,-1) ellipse (.4 and 1.4); 
        \draw [rounded corners=8pt] (1.15,-0.85)--(1.15,-1.2)--(0.16,-1.2)--(-0.2,-1.15)--(-0.2,0.13)--(0.12,0.16)--cycle;
         \draw [rounded corners=8pt] (1.25,-1.2)--(1.25,-0.8)--(0.16,-0.8)--(-0.15,-0.85)--(-0.15,-2.15)--(0.18,-2.15)--cycle;
         
         \end{tikzpicture}
}

\end{center}

	\caption{The $3$-uniform hypergraph $F_5$ and $K_4^-$.}\label{figF5}
\end{figure}

The {\em Tur\'{a}n hypergraph} $T_{r}(n)$ is the complete $n$-vertex $r$-uniform 
$r$-partite hypergraph whose partite sets are as equally-sized as possible. In particular, Mantel's Theorem states that the maximum triangle-free graph on $n$ vertices is $T_2(n)$.
{Finding extremal graphs for $3$-uniform hypergraphs is much more difficult; even  the extremal  hypergraph of $K_4^-$ is not known. The best known construction is due to Frankl and F\"{u}redi~\cite{k4-}, with asymptotically $n^3/21$ edges,  so a maximum $K_4^-$-free $3$-uniform hypergraph is not tripartite.} 
One of the first  extremal results in extremal hypergraph theory is determining the extremal hypergraph of $F_5$. 
This problem $F_5$ was first considered by Bollob{\'a}s~\cite{MR0345869}, who proved results for cancellative hypergraphs, i.e., that the maximum $\{K_4^-, F_5\}$-free hypergraph is {\tpt}.
Frankl and F\"{u}redi~\cite{MR729787} proved that 
the maximum $3$-uniform hypergraph on $n$
vertices containing no copy of $F_5$
is $T_3(n)$ for $n>3000$.
The hypergraph $F_5$ is the smallest $3$-uniform hypergraph whose extremal hypergraph is $T_3(n)$.  

Our main result is a random variant of the theorem of Frankl and F\"{u}redi~\cite{MR729787}, i.e., that for sufficiently large $p$ {a} largest 
$F_5$-free subgraph of $G^3(n,p)$ is w.h.p.~tripartite, and our $p$ is close to best possible.  
\begin{thm}\label{thm:main}
There exists a positive constant $K$ such that w.h.p.~the following is true. If {$\G$} $= G^3(n,p)$ is a $3$-uniform random hypergraph with $p>K\log n/n$, then every
maximum   {(with respect to the number of edges)}  $F_5$-free subhypergraph of {$\G$} is {\tpt}.
\end{thm}

{If $p$ is very  small, then an $F_5$-free subhypergraph of $\G$ is also w.h.p.~tripartite  since $\G$ itself is likely to be {\tpt}, so this case is not so interesting for us.}

If $p=0.1\sqrt{\log n}/n$, then w.h.p.\@ there is a maximum $F_5$-free subhypergraph of $G^3(n,p)$ that is not {\tpt}.  
 To see this, first  using the  second moment method  one can  prove that w.h.p.\@ there are $n/5$ vertex disjoint copies of $K_4^-$ in $G^3(n,p)$.  Then using Janson's inequality (which computation is quite delicate, we omit the details), one can prove that w.h.p.\@  one of them has the property that none of its  edges are  in an $F_5$. 
Then a maximum $F_5$-free subhypergraph of $G^3(n,p)$ would contain this $K_4^-$, and note that  $K_4^-$ is not {\tpt}.
 
   We consider that the above reasoning  might be optimal, therefore we conjecture that $\sqrt{\log n}/n$ is the correct order of $p$.

\bcj
There exists a positive constant $K$ such that w.h.p.\@ the following is true. If {$\G$} $= G^3(n,p)$ is a $3$-uniform random hypergraph with $p>K\sqrt{\log n}/n$, then every maximum  $F_5$-free subhypergraph of {$\G$} is {\tpt}.
\ecj

Note that a weaker result than Theorem~\ref{thm:main} appeared in the thesis of the second author~\cite{Butter}. 
To improve the results of~\cite{Butter},  some ideas of~\cite{DK}, see Lemma~\ref{lem:key}, are used in this paper, but there are several differences as well. 
Our result, similar to~\cite{DK}, characterizes the precise structure of the extremal subgraph of the random hypergraph. 

Recently, powerful 
 general  asymptotic  statements were proved about extremal substructures of random discrete structures, see 
  Balogh--Morris--Samotij~\cite{bms}, Conlon--Gowers~\cite{ConlonGowers},  Samotij~\cite{wojtek}, Saxton--Thomason~\cite{SaxThom} and Schacht~\cite{Schacht}.


We shall use Theorem 1.8 of Samotij~\cite{wojtek}, which transferred a stability theorem of Keevash and Mubayi~\cite{kevmub}:

\begin{thm}\label{makeF5tripartite}
For every $\delta > 0$ there exist positive constants $K$ and $\eps$ such that if $p_n \geq K/n$, then w.h.p.\@~the following holds.
Every $F_5$-free subgraph of $G^3(n,p_n)$ with at least $(2/9 - \eps){n \choose 3} p_n$ edges admits a partition  $\left( A_1, A_2, A_3 \right)$ of $[n]$ such that all but at most $\delta n^3p_n$ edges have one vertex in each $A_i$. 
\end{thm}

The hypergraph $F_5$ is an example of what Balogh, Butterfield, Hu, Lenz, and Mubayi~\cite{chromthres} 
call a ``critical hypergraph''; they proved that if $H$ is a critical hypergraph, then for sufficiently 
large $n$ the unique largest $H$-free hypergraph with $n$ vertices is the Tur\'{a}n hypergraph. 
We could prove results analogous to Theorem~\ref{thm:main} for the family of critical hypergraphs, 
as some ideas of our proofs are from~\cite{chromthres}, 
but this extension to critical hypergraphs is likely to be very technical, and probably we would not be able to determine the whole range of $p$ where the sparse extremal theorem is valid.

The rest of the paper is organized as follows. In
Section~\ref{F5_tools} we introduce some more notation and state some standard properties of $G^3(n,p)$. In Section~\ref{KeyLemma} we provide our main lemmas and prove them. We  prove our main result, Theorem~\ref{thm:main}, in Section~\ref{Proof}.
To simplify the formulas, we shall often omit floor and ceiling signs when they are not crucial.

\section{Notations and Preliminaries}\label{F5_tools}

For the remainder of the paper, $\G$ will always denote the $3$-uniform random hypergraph $G^3(n,p)$. 
The \emph{size} of a hypergraph (graph) $\HH$, denoted $|\HH|$, is the number of hyperedges {(edges)} it 
contains.  We denote by $\t(\G)$ the  size of a largest {\tpt} subhypergraph of $\G$. 

We write $x = (1\pm \eps) y$ when $(1-\eps) y\le x\le (1+\eps)y$. We say $\Pi = (A_1, A_2 ,A_3)$ is a \emph{balanced partition} if $|A_i| = (1\pm 10^{-10})n/3$ for all $i$.   Given a partition $\Pi = (A_1, A_2 ,A_3)$ and 
a $3$-uniform hypergraph $\HH$, we say that an edge $e$ of $\HH$ is \emph{crossing} if $e\cap A_i$ is non-empty for every $i$.   
We use $\HH[\Pi]$ to denote the set of crossing edges of $\HH$.

The  \emph{link graph} $L(v)$ of a vertex $v$ in $\G$ is the graph  {induced by the} edge set $\{xy: xyv\in \G\}$.  
The  \emph{crossing link graph} $L_{\Pi}(v)$ of a vertex $v$ is the subgraph of $L(v)$ whose edge set is $\{xy: xyv$ is a crossing edge of $\G\}$. 
The \emph{degree} $d(v)$ of $v$ is the size of $L(v)$ (i.e.~the number of edges in $L(v)$), and the \emph{crossing degree} $d_{\Pi}(v)$ of $v$ is the size of $L_{\Pi}(v)$.
The  \emph{common link graph} $L(u,v)$ of two vertices $u$ and $v$ 
is $L(u)\cap L(v)$ and the \emph{common degree} $d(u,v)$ is the size of $L(u,v)$.
The \emph{common crossing link graph} $L_{\Pi}(u,v)$ of two vertices $u$ and $v$ 
is $L_{\Pi}(u)\cap L_{\Pi}(v)$ and the \emph{common crossing degree} $d_{\Pi}(u,v)$ is the size of $L_{\Pi}(u,v)$. Given two vertices $u$ and $v$, their \emph{co-neighborhood} $N(u,v)$ is $\{x: xuv\in \G\}$; the \emph{co-degree}
of $u$ and $v$ is the number of vertices in their co-neighborhood.

Given two disjoint sets $A$ and $B$, we use $[A, B]$ to denote the set $\{a\cup b: a\in A, b\in B\}$. 
We will use this notation in two contexts.  First, if both $A$ and $B$ are vertex sets, then $[A,B]$ is a complete bipartite graph. 
Second, if $A$ is a subset of a vertex set $V$ and $B$ is a set of \emph{pairs} of vertices of $V\setminus A$, then $[A,B]$ is a $3$-uniform hypergraph.  
In these two contexts, given a graph or hypergraph $\HH$, let $\HH[A,B]$ denote the set $\HH\cap [A,B]$.  
Note that in the first case $\HH[A,B]$ is the bipartite subgraph of $\HH$ induced by $A$ and $B$.  
In the second case, $\HH[A,B]$ is the $3$-uniform subhypergraph of $\HH$ whose edges have exactly one vertex in $A$ and contain a pair of vertices from $B$. 

We say a vertex partition $\Pi$ with three classes, 
which we will call a \emph{$3$-partition}, is {\em maximum} if $|\G[\Pi]| = \t(\G)$. Let $\FF$ be a maximum $F_5$-free subhypergraph of $\G$. Clearly $\t(\G)\le |\FF|$. To prove Theorem~\ref{thm:main}, we will show that w.h.p.\@ $ |\FF|\le \t(\G)$ is also true for an appropriate choice of $p$. 
Moreover, we will prove that if $\FF$ is not {\tpt}, then w.h.p.\@ $|\FF| < \t(\G)$.

We will make use of the following Chernoff-type bound  to prove Propositions~\ref{prop:codegree}-\ref{prop:bad_degree}, which state useful properties of $G^3(n,p)$. The proofs of those propositions are standard applications of the Chernoff bound,  therefore we omit some of them.

\begin{lem}\label{chernoff}
Let $Y$ be the sum of mutually independent indicator random variables, and let $\mu = E[Y]$. 
For every $\eps > 0$, 
	\[
		P[|Y - \mu| > \eps\mu] < 2 e^{-c_\eps \mu}, 	
	\]
where $c_\eps = \min \{ -\ln\left(e^{\eps}(1+\eps)^{-(1 + \eps)}\right), \eps^2/2 \}$. 
\end{lem}

For the rest of this paper, we always use $c_{\eps}$ (which depends on $\eps$) to denote the constant in Lemma~\ref{chernoff}.

\begin{prop}\label{prop:codegree}
For any $\eps>0$, there exists a constant $K$ such that if $p>K\log n/n$, then
w.h.p.\@ the co-degree of any pair of vertices in $\G$ is $ (1\pm\eps)pn$.
\end{prop}

\begin{prop}\label{prop:common_degree}
For any $\eps>0$, there exists a constant $K$ such that if $p>K\sqrt{\log n}/n$, then w.h.p.\@ the common degree $d(x,y)$ of any pair of vertices $(x,y)$ in $\G$ is $(1\pm \eps)p^2n^2/2$.
\end{prop}

\begin{prop}\label{prop:degree}
For any $\eps>0$, there exists a constant $K$ such that if $p>K\log n/n^2$, then w.h.p.\@ for any vertex $v$ of $\G$, we have $d(v) = (1\pm\eps)pn^2/2$.
\end{prop}

\begin{prop}\label{prop:cross_degree}
For any $\eps>0$, there exists a constant $K$ such that if $p>K/n$, then w.h.p.\@ for any $3$-partition $\Pi = (A_1, A_2 ,A_3)$ with $|A_2|, |A_3|\ge n/20$ and any vertex $v\in A_1$, we have $d_{\Pi}(v) =  (1\pm \eps)p|A_2||A_3|$. 
\end{prop}

%
%

For a vertex $v$ and a vertex set $S$, let $\E$ be a subset of $\{vxw\in\G: x\in S\}$ satisfying that for every 
$\ x\in S,$ there exists an $ e\in \E$ such that $x\in e$, and let  $T$ be a subset of $L(v)$. Define 
$$\K_{v, \E}[S, T] = \{xyz: x\in S, yz \in T, \exists  e \in \E \textnormal{ s.t.\@ }x\in e,\  \ y,z\notin e\},$$ 
and    $\G_{v, \E}[S, T]:=  \K_{v, \E}[S, T] \cap  \G.$ Observe, since  $\G$ is the random hypergraph, we have 
$$\EE [\left|\G_{v, \E}[S, T]\right|] = p\left|\K_{v, \E}[S, T]\right|.$$

Then for any $xyz\in \G_{v, \E}[S, T]$ with $x\in S,yz\in T$, we can find an $F_5=\{vxw,vyz,xyz\}$ where $vxw\in \E$. The condition $y,z\notin e$ in the definition of $\G_{v, \E}[S, T]$ guarantees that we can find such an $F_5$ instead of only a $K_4^-$. The somewhat artificial definition of $\G_{v, \E}[S, T]$ is needed for given   $v$, $\E$, $S$ and $T$ to forbid  many hyperedges, which could create an $F_5.$

\begin{prop}\label{prop:edges}
For any constants $\eps, \eps_1,\eps_2>0$, there exists a constant $K$ such that if $p>K\log n/n$, then w.h.p.\@ for every choices of $\{v,S,\E, T\}$ as above with $|S|\ge \eps_1n$ and $|T|\ge \eps_2pn^2$,
we have $|\G_{v,\E}[S, T]| = (1\pm\eps)p|S||T|$.
\end{prop}

\begin{proof}
Note that here we first reveal the edges containing $v$; given this choice we fix $S$, $T$ and $\E$. We also assume that the conclusions of the previous propositions hold.
For $x\in S$, let $d_{\E}(x) = |\{e\in \E: x\in e\}|$ and $T_x = \{ yz\in T : vxy \in \E\}$. If $d_{\E}(x) > 2$, then clearly $[x, T]\subseteq\K_{v, \E}[S, T]$. If $d_{\E}(x) \le 2$, then by Proposition~\ref{prop:codegree}, we have $|T_x| \le 2\cdot 2pn = 4pn$, and  clearly $[x, T\setminus T_x]\subseteq\K_{v, \E}[S, T]$. Therefore, $$\left|[S, T]\right| - \left| \K_{v, \E}[S, T] \right|\le \sum_{x\in S, d_{\E}(x) \le 2} |T_x| \le |S|\cdot 4pn.$$
 We have $\left|[S, T]\right| = |S||T| \ge |S_1|\eps_2 pn^2$, so $\left| \K_{v, \E}[S, T] \right| = (1-o(1))|S||T|$. Let $\mu = E[\left|\G_{v, \E}[S, T]\right|] = p\left|\K_{v, \E}[S, T]\right| =  (1-o(1))p|S||T|$. By Lemma~\ref{chernoff} we have 
\[
\P\left[ \left| \left|\G_{v, \E}[S, T]\right| - \mu \right|  >\eps \mu  \right] < 2e^{-c_{\eps}\mu } . 
\]

For sets $S, T$ with $|S|=s$ and $|T|=t$, we have at most $n$ choices for $v$, $\binom{n}{s}$ choices for $S$, $2^{2spn}$ choices for $\E$ and $\binom{pn^2}{t}$ choices for $T$. Then by the union bound, the probability that the statement of Proposition~\ref{prop:edges} does not hold is bounded by 
\[
\sum_{s\ge \eps_1n}\sum_{t\ge \eps_2pn^2} n\binom{n}{s}2^{2spn}\binom{pn^2}{t}2e^{-c_{\eps}\mu } \le 
\sum_{s\ge \eps_1n}\sum_{t\ge \eps_2pn^2} n\binom{n}{s}2^{2spn}\binom{pn^2}{t}2e^{-c_{\eps}st/2 } = o(1).\qedhere
\]

\end{proof}

{ The proof of the  following proposition is based on Theorem~\ref{makeF5tripartite}.}

\begin{prop}\label{prop:f5free}
Let $\delta$ be a small positive constant. Then there is an $\eps>0$ and a large constant $K=K(\delta,\eps)$ such that  if $p>K/n$ and
$\FF$ is a maximum $F_5$-free subhypergraph of $\G$, then $|\FF|\ge (2/9-\eps)\binom{n}{3}p$. Furthermore, for every 
$F_5$-free subhypergraph $\FF$  of $\G$ 
and  a $3$-partition  $\Pi$  maximizing $|\FF[\Pi]|$, where $|\FF|\ge (2/9-\eps)\binom{n}{3}p$, the partition $\Pi= \left( A_1, A_2, A_3 \right)$ is  balanced and all but at most $\delta n^3p$ edges have one vertex in each $A_i$. 
\end{prop}

\begin{proof}
We may assume that  $\delta, \eps < 10^{-100}$, where $\eps$ is determined by  Theorem~\ref{makeF5tripartite} given $\delta$.
For a partition $\Pi = (A_1,A_2,A_3)$, 
Proposition~\ref{prop:cross_degree} implies that w.h.p.\@ $|\G[\Pi]| = (1\pm\eps)p|A_1||A_2||A_3|$ if $|A_2|,|A_3|\ge n/20$. Clearly $|\FF|\ge t(\G) \ge |\G[\Pi]|$. If $|A_1|=|A_2|=|A_3| = n/3$, then we have $|\FF|\ge (2/9-\eps)\binom{n}{3}p$. Theorem~\ref{makeF5tripartite} implies that if $\Pi$ maximizes $|\FF[\Pi]|$, then $\G[\Pi]\ge \FF[\Pi] \ge (2/9-\eps-\delta)\binom{n}{3}p$, and the number of non-crossing edges is at most $\delta p n^3$.

If $\Pi$ is not balanced and $|A_2|,|A_3|\ge n/20$, then $|\G[\Pi]|\le (1+\eps)p|A_1||A_2||A_3| < (2/9-2\eps)\binom{n}{3}p$. If $\Pi$ is not balanced and one of $|A_1|,|A_2|,|A_3|$ is less than $n/20$, then Proposition~\ref{prop:degree} implies that $|\G[\Pi]|< n/20 \cdot (1+\eps)pn^2/2 < (2/9-\eps-\delta)\binom{n}{3}p$. Therefore, if $\Pi$ maximizes $|\FF[\Pi]|$, then $\Pi$ is balanced.
\end{proof}

Given a balanced partition $\Pi = (A_1, A_2, A_3)$, let $Q(\Pi) = \{(u,v)\in \binom{A_1}{2}: d_{\Pi}(u,v)<0.8 n^2p^2/9\}$.  In words, $Q(\Pi)$ is the set of pairs of vertices in $A_1$ that have low common crossing degree (in $\G$).

\begin{prop}\label{prop:bad_degree}
There exists a constant $K$ such that if $p > K/n$, then w.h.p.\@ for every balanced {partition} $\Pi$ and every vertex $v$, we have $d_{Q(\Pi)}(v) < 0.001/p$. 
\end{prop}

\begin{proof}
Let $\eps = 0.1$. By Proposition~\ref{prop:cross_degree}, we may assume that $d_{\Pi}(v) \ge (1-\eps)pn^2/9$, and therefore, $d_{\Pi}(u,v)\le \frac{0.8}{1-\eps} d_{\Pi}(v)p$ for $(u,v)\in $ {$Q(\Pi)$}.

If a vertex $v$ and a balanced {partition} $\Pi$ violate the statement of Proposition~\ref{prop:bad_degree}, then there are $S \subseteq V$ and $T = L_{\Pi}(v)$ with $|S| :=s= \left\lceil 0.001/p\right\rceil$ and $\left|\G[S, T]\right| \le \frac{0.8}{1-\eps} |S||T|p$. We have at most $3^n$ choices of $|\Pi|$, $n$ choices of $v$, $\binom{n}{s}$ choices of $S$, so the probability of such a violation is at most
\[
3^nn \binom{n}{s}\exp(-c\cdot 0.001/p\cdot pn^2\cdot p)=o(1),
\]
where $c$ is some small constant.
\end{proof}

The following statement is heavily used in the proof of Lemma~\ref{lem:f5}, which is one of the two main lemmas we use to prove our main theorem.

\begin{lem} \label{badpairs}
Let $a$ and $r$ be positive integers. 
For any $\eps>0$, there exists a constant $K$ such that if 
 $p > K\log n /n, a\le \eps n$ and  
 \begin{equation}\label{eq:badpairs}
{n \choose a} \cdot {n^2/2 \choose r} \cdot \exp(- c_{1}\eps  p n r) = o(1),
 \end{equation}
then w.h.p.\@ the following holds.  For any set 
of vertices $A$ with $|A| = a$, there are at most $r$ pairs $\{u,v\} \in {V(G){\setminus A} \choose 2}$
such that $|N(u,v) \cap A| > 2\eps p n$. 
\end{lem}

\begin{proof}
Fix a set $A$ of size $a$.  We shall show 
that there are at most $r$ pairs $u,v$ in $ {V(G)\setminus A \choose 2}$ for which $|N(u,v) \cap A|$ is large.  
For each pair of vertices $u$ and $v$, let $B(u,v)$ be the event that $|N(u,v) \cap A| > 2\eps p n\ge 2p a $. 
By (a  slight variant of) Chernoff's inequality, 
\[\P[B(u,v)] < e^{-c \eps p n }\]
for $c=c_1$ in Lemma~\ref{chernoff}.
If $\{u,v\} \neq \{u',v'\}$ then $B(u,v)$ and $B(u',v')$ 
are independent events. Consequently, 
the probability that $B(u,v)$ holds for at least $r$ pairs is at most 
\[
{n^2/2 \choose r}e^{-c \eps p n r}.
\]
There are ${n \choose a }$ choices of $A$. 
 Therefore, if~\eqref{eq:badpairs} holds, 
then w.h.p.\@ there are at most $r$ pairs $\{u,v\} \in {V(G){\setminus A} \choose 2}$ 
such that $|N(u,v) \cap A| > 2\eps p n$. 
\end{proof}

\section{Key Lemmas for Theorem~\ref{thm:main}}\label{KeyLemma}
Let $\FF$ be an $F_5$-free subhypergraph of $\G$; we want to show that $|\FF|\le t(\G)$. The following lemma proves this with some additional conditions on $\FF$.
The \emph{shadow graph} of a $3$-uniform hypergraph $\HH$ is the graph
with $xy$ an edge if and only if 
     there exists some hyperedge of $\HH$ that contains both
      $x$ and $y$.

\begin{lem}\label{lem:f5}
Let $\FF$ be an $F_5$-free subhypergraph of $\G$ and $\Pi = (A_1, A_2, A_3)$ be a balanced partition maximizing $|\FF[\Pi]|$. Let $\B_i = \{e\in \FF : |e\cap A_i| \ge 2\}$  for $1 \leq i \leq 3$.
 There exist positive constants $K$ and $\delta$  such that if $p>K\log n/n$ and if the following conditions hold: 
\vspace{-2ex}
\begin{enumerate}[(i)]
  \setlength{\itemsep}{1pt}
  \setlength{\parskip}{0pt}
  \setlength{\parsep}{0pt}
\item\label{lem:f5:1} $\sum_i |\B_i| \le \delta pn^3$,
\item\label{lem:f5:3} the shadow graph of $\B_1$ is disjoint from $Q(\Pi)$, 
\end{enumerate}
\vspace{-1ex}
then
w.h.p.\@ $|\FF[\Pi]| + 3|\B_1| \le |\G[\Pi]|$, where equality is possible only if $\FF$ is tripartite.
\end{lem}

If $\FF$ is a maximum $F_5$-free subhypergraph of $\G$ and not tripartite, then
by Proposition~\ref{prop:f5free}, for every $\delta>0$, w.h.p.\@ Condition~\eqref{lem:f5:1} of Lemma~\ref{lem:f5} holds.
Without loss of generality, we may assume that $|\B_1|\ge |\B_2|, |\B_3|$, in particular $\B_1\ne\emptyset$.  {If} $Q(\Pi)=\emptyset$, then  Condition~\eqref{lem:f5:3}  of Lemma~\ref{lem:f5}  holds and we  get $|\FF| < t(\G)$, a contradiction. 
If $Q(\Pi)=\emptyset$ for every balanced partition $\Pi = (A_1, A_2, A_3)$, then the proof would be completed. Unfortunately, we are only able to prove this property for $p>K/\sqrt{n}$ with 
some large $K$, so Lemma~\ref{lem:f5} implies that Theorem~\ref{thm:main} is true for $p>K/\sqrt{n}$. To improve the bound on $p$ from the order of $1/\sqrt{n}$ to $\log n/n$, 
 we prove that  $Q(\Pi) = \emptyset$ for every maximum $3$-partition $\Pi$. {(Recall that a $3$-partition $\Pi$ is maximum if $\t(\G)=|\G[\Pi]|$.)}  This is stated in the following lemma, which says that if $Q(\Pi)\ne\emptyset$, then $\Pi$ is far from being a maximum $3$-partition. The proof of Lemma~\ref{lem:key} is along the lines of the proof of Lemma~5.1 in DeMarco--Kahn~\cite{DK}.


\begin{lem}\label{lem:key}
There exist positive constants $K$ and $\delta$  such that if $p>K\log n/n$ and  the $3$-partition  $\Pi$ is balanced,  then 
\whp
\[
\t(\G)  \ge  |\G[\Pi]|  + |Q(\Pi)|\delta n^2p^2,
\]
where equality is possible only if $Q(\Pi) = \emptyset$.
\end{lem}

Note that if  $Q(\Pi) = \emptyset$, then by definition  $\t(\G) \ge |\G[\Pi]|$. 
We will use Lemmas~\ref{lem:f5} and~\ref{lem:key} to prove Theorem~\ref{thm:main}.
In the next two subsections we prove these two lemmas.


\subsection{Proof of Lemma~\ref{lem:f5}}
We will begin with a sketch of the proof of Lemma~\ref{lem:f5}, which will motivate the following lemmas. 
%

Let $\M$ be the set of crossing edges of $\G\setminus\FF$, and assume that $|\B_1|\ge |\B_2|, |\B_3|$. 
{ If $|\B_1|=0$, then we have $|\FF[\Pi]| + 3|\B_1| \le |\G[\Pi]|$, so}
to prove Lemma~\ref{lem:f5}, it suffices to prove that { if $\B_1\neq\emptyset$, then} $3|\B_1|< |\M|$. So we assume for contradiction that $ |\M|\le 3|\B_1|\le 3\delta pn^3$, where the second inequality follows from Condition~\eqref{lem:f5:1} of Lemma~\ref{lem:f5}.

For each edge $e=w_1w_2w_3 \in \B_1$ with $w_1, w_2 \in e \cap A_1$, 
because $w_1w_2 \notin Q(\Pi)$, 
there exist at least $0.8 p^2 n^2/9$ choices of $y \in A_2$ and $z \in A_3$
such that $w_1yz$ and $w_2yz$ are both crossing edges of $\G$. By Proposition~\ref{prop:codegree}, the co-degree of $w_1$ and $w_3$ is w.h.p.\@ at most $2pn$. Therefore, there are at least $0.8 p^2 n^2/9- 2pn \ge p^2n^2/12$ choices of such pairs $(y,z)$ such that  $w_3\notin \{y,z\}$, and then each of these pairs $(y,z)$ together with $e$ form a copy of $F_5=\{w_1w_2w_3, w_1yz, w_2yz \}$ in $\G$.
Since $\FF$ contains no copy of $F_5$, at least one of 
$w_1yz$, $w_2yz$ must be in $\M$. 

We will count elements of $\M$ by counting the embeddings of $F_5$ in $\G$ that contain 
some $e \in \B_1$. Each such $F_5$ contains at least one edge from $\M$, and this will provide a lower bound on the size of $\M$ in terms of $|\B_1|$. 
Instead of counting copies of $F_5$ itself, we will count copies of $\F_5$ which is a 4-set $\{w_1,w_2,y,z\}$ such that there 
exists $e\in \B_1$ with $w_1, w_2 \in e \cap A_1,~y,z\notin e$ and $w_1yz, w_2yz$ being crossing edges. 
It is easy to see that each $\F_5$ yields {at least one} $F_5$ containing some $e\in\B_1$. 
So for each pair $w_1,w_2\in e\cap A_1$ for any $e\in \B_1$, there are at least $p^2n^2/12$ copies of $\F_5$ containing $w_1,w_2$.
We will count copies of $\F_5$ in $\G$ by considering several cases, based on the relative sizes of 
the sets $C_1$ and $C_2$, defined below. 

Let 
\newcommand{\cdefineC}{\eps_1}
\newcommand{\cmincross}{\eps_2}
\newcommand{\cdeltal}{\eps_3}
\begin{equation}\label{constants}
\cdefineC = \frac{1}{960}, \qquad \cmincross = \frac{1}{400}, \qquad \delta = \frac{\cdefineC^2 \cmincross}{108\cdot 160}\Text{and}
   \cdeltal = \frac{108 \delta}{\cdefineC}.
\end{equation}  

Denote by {$\sg$} the shadow graph of $\B_1$ on the vertex set $A_1$. 
Let $C = \left\{ x \in A_1 : d_\sg(x) \geq \cdefineC n \right\}$, 
 $C_1$ be the set of every vertex in  $C$ that is in at least $\eps_2 pn^2$ crossing edges of $\FF$ and let $C_2 = C \setminus C_1$.

With these definitions in hand, we are prepared to prove the following lemmas, which will lead 
to a proof of Lemma~\ref{lem:f5} at the end of this subsection. 

\begin{lem}\label{boundonC}
$|C| \le \cdeltal n$.
\end{lem}

\begin{proof}
{For} each edge $wx \in E(\sg)$, since $wx\notin Q(\Pi)$, there are at least
$p^2n^2/12$ choices of $y \in A_2$, $z \in A_3$ such that 
$\{w,x,y,z\}$ spans an $\F_5$ in $\G$. Then
$xyz , wyz \in \G$ and
one of these two edges must be in $\M$, otherwise $\FF$ contains a copy of $F_5$.
On the other hand, by Proposition~\ref{prop:codegree} with $\eps=0.5$, for each edge $xyz\in \M$ with $x\in A_1, y\in A_2, z\in A_3$, there are at most $3pn/2$ choices of $w$ such that $wyz\in \G$.
{Therefore, we have} $\frac{3}{2}pn|\M| \geq${$|\sg|$}$p^2n^2/12$.
We assume  $3\delta pn^3 \ge  |\M|$, so
$\frac{3}{2}pn \cdot 3\delta pn^3 \ge \frac{3}{2}pn|\M|$. It follows that $54\delta n^2 \ge { |\sg|}$.

Now, every vertex in $C$ has degree at least $\cdefineC n$  in $\sg$, so $\eps_1 n |C| \le 2{ |\sg|}\le 108\delta n^2$ 
implies that $|C| \le 108 \delta \eps_1^{-1}n = \cdeltal n$.
\end{proof}

\begin{lem} \label{boundMfromC1}
$|\M| \geq 20 p n^2 |C_1|.$
\end{lem}

\begin{proof}
Assume $|C_1|\ge 1$, otherwise this inequality is trivial.
For each $x \in C_1$, let $T_x := \{(y,z) \in A_2 \times A_3 \,|\, xyz \in \FF\} \subseteq L_{\Pi}(x)$.  By the definition of $C_1$, we have $|N_\sg(x)| \ge \eps_1 n$ and $|T_x| \ge \eps_2 p n^2$ for each $x\in C_1$. 
We will count the number of copies of $\F_5:\{x,w,y,z\}$ in $\G$ with $x\in C_1, w\in N_\sg(x), xyz\in \FF$ and $wyz\in \G$.  
 By Proposition~\ref{prop:edges} with $v=x, \ S = N_{\sg}(x), \ \E = \{ {\new e}\in \B_1:x\in {\new e}\}$ and $T = T_x$, 
 there are at least 
$\frac{1}{2}d_\sg(x)|T_x|p$ such copies of $\F_5$ for each $x\in C_1$.
Therefore, the total number of such copies of $\F_5$ is at least 
\begin{equation}\label{eq:boundMfromC1}
\sum_{x\in C_1} \frac{1}{2}d_\sg(x)|T_x|p \geq \frac{1}{2}|C_1| \cdot \eps_1 n \cdot \eps_2 pn^2 \cdot p= 
\frac{1}{2} \eps_1\eps_2 p^2n^3 |C_1|. 
\end{equation}

Say that an edge $wyz \in \M$ is \emph{bad} if $w \in A_1$, $y \in A_2$, $z \in A_3$, and there are at least 
$ 2 \cdeltal p n$ vertices $x \in C_1$ for which $xyz\in\G$. 
Because 
$|C_1| \leq |C|$, which by Lemma~\ref{boundonC} has size at most $\cdeltal n$,  
we can apply 
Lemma~\ref{badpairs} with $\eps = \cdeltal, a =\eps n, r =  (\log\log n) / p$ and $A = C_1$ to 
show that
there are at most $(\log\log n)/p$ pairs $(y,z)\in A_2\times A_3$ that are in some bad edge.
By Proposition~\ref{prop:codegree}, the co-degree of each such pair $(y,z)$ is at most
$2pn$. Therefore, each $(y,z)$ is in at most ${2pn\choose 2}$ $\F_5$'s, and so the number of copies of $\F_5$ estimated in~\eqref{eq:boundMfromC1} that contain a non-bad edge from $\M$ is at least
\[
\frac{1}{2} \eps_1\eps_2p^2n^3|C_1| - {2pn\choose 2} \cdot \frac{\log\log n}{p}.
\]
 Now, 
\[
 {2pn\choose 2}\cdot  \frac{\log\log n}{p}  \le  2pn^2 \log\log n 
  											 \le  \frac{1}{4} \eps_1\eps_2p^2n^3 
  											 \le  \frac{1}{4} \eps_1\eps_2p^2n^3|C_1|,
\]
where the second inequality follows from $p\ge \log n/n$.
Therefore, at least 
\[
\frac{1}{2}\eps_1\eps_2 p^2 n^3 |C_1| - \frac{1}{4}\eps_1\eps_2 p^2 n^3 |C_1| = \frac{1}{4}\eps_1\eps_2 p^2 n^3 |C_1| 
\]
of the copies of $\F_5$ estimated in~\eqref{eq:boundMfromC1} contain a non-bad edge from $\M$. Each such  non-bad edge from $\M$ is contained in at most 
$2\cdeltal pn = 216 \delta \eps_1^{-1}p n$ such copies of $\F_5$, and so 
\[
|\M| \geq \frac{\eps_1^2\eps_2 p^2 n^3 |C_1| }{4\cdot 216 \delta p n} = \frac{\eps_1^2 \eps_2}{8 \cdot 108 \delta}p n^2 |C_1| = 20 pn^2 |C_1| .
\]
\end{proof}

\begin{lem} \label{subgraphwithlargemindeg}
If $\sg'$ is a subgraph of $\sg$ such that $\Delta(\sg') \leq \cdefineC n$, then
\begin{align*}
\left| \M \right| \geq 
20 p n {\left| \sg' \right|}.
\end{align*}
\end{lem}

\begin{proof}
For each $wx \in E(\sg')$, since $wx\notin Q(\Pi)$, there are at least $p^2n^2/12$ choices of 
$(y,z) \in A_2 \times A_3$ such that $\{w,x,y,z\}$ spans an $\F_5$ in $\G$. 
There are therefore at least $ \frac{1}{12}|\sg'|p^2n^2$ copies of $\F_5$, and at least one of $wyz, xyz$ must be in $\M$ for 
each of these copies of $\F_5$. 

Consider an edge $xyz \in \M$ with $x \in V(\sg')$.  We will count how many of these copies of $\F_5$ in 
$\G$ contain $xyz$. Say that $xyz$ is \emph{bad} if there exist at least $2\cdefineC p n$ vertices 
$w \in N_{\sg'}(x)$ with $wyz \in \G$.  For each $x \in V(\sg')$, let $d_x = d_{\sg'}(x)$ and denote by $r_x$ the number of pairs $(y,z)$ such that $xyz$ is bad. By Proposition~\ref{prop:codegree}, the co-degree of each such pair $(y,z)$ is at most $2pn$, so there exist at most $\min\{2pn, d_x\}$ vertices  $w \in N_{\sg'}(x)$ with $wyz \in \G$.
Then the number of copies of $\F_5$ that contain a non-bad edge from $\M$ is at least 
\begin{equation}\label{eq:badP}
\frac{1}{2}\sum_{x{\in V(\sg')}} d_x\frac{p^2 n^2}{12} - \sum_{x{\in V(\sg')}}  r_x\cdot\min\{2pn, d_x\}.
\end{equation}

 We will prove $\frac{1}{2} d_x\cdot p^2 n^2/12 \ge 2r_x\cdot\min\{2pn, d_x\}$ for every vertex $x{\in V(\sg')}$ by applying Lemma~\ref{badpairs} with $\eps = \cdefineC, A= N_{\sg'}(x)$ and various choices of $a$ and $r$ depending on $d_x$. Note that 
$d_x\le\Delta(\sg') \leq \cdefineC n$. So $d_x$ will fall into one of the following three cases.

\begin{enumerate}
  \setlength{\itemsep}{1pt}
  \setlength{\parskip}{0pt}
  \setlength{\parsep}{0pt}
\item $d_x > 2pn$ and $\frac{\log n}{p^2n}\le d_x\le \cdefineC n$. We apply Lemma~\ref{badpairs} with $a = \cdefineC n$ and $r = (\log\log n)/p$ to obtain that $r_x \le (\log\log n )/p$.
\item $d_x > 2pn$ and $\frac{\log n}{p^{k+2}n^{k+1}}\le d_x\le \frac{\log n}{p^{k+1}n^{k}}$ for some integer $k\in [1, \frac{\log n}{\log \log n}]$. We apply Lemma~\ref{badpairs} with $a  = \frac{\log n}{p^{k+1}n^{k}}$ and $r = a/100$ to obtain that $r_x \le \frac{\log n}{100p^{k+1}n^{k}}\le pnd_x / 100$. 
\item $d_x\le 2pn$.  We  apply Lemma~\ref{badpairs} with $a=2pn$ and $r=p^2n^2/50$ to obtain that $r_x\le p^2n^2/50$.
\end{enumerate}

As long as $a$ and $r$ are positive integers and satisfy~\eqref{eq:badpairs}, we can apply Lemma~\ref{badpairs}.
For each of these three cases, we can easily check that 
$$\frac{1}{2} d_x\frac{p^2 n^2}{12} \ge 2 r_x\cdot \min\{2pn, d_x\}.$$ 
Therefore, the number of copies of $\F_5$  estimated in~\eqref{eq:badP} is at least
\begin{equation*}
\frac{1}{2}\sum_{x\in V(\sg')} d_x\frac{p^2 n^2}{12} - \sum_{x{\in V(\sg')}} r_x \cdot \min\{2pn, d_x\} \ge \frac{1}{4}\sum_{x\in V(\sg')} d_x\frac{p^2 n^2}{12} = \frac{1}{24}|\sg'|p^2n^2.
\end{equation*}
By definition, an edge that is not bad is in at most $2\cdefineC p n$ of the copies of $\F_5$ estimated in~\eqref{eq:badP}.  
Therefore, 
\[
|\M| \geq \frac{1}{24}\cdot\frac{|\sg'|p^2n^2}{2\cdefineC p n} = \frac{1}{48\cdefineC}\cdot pn|\sg'| = 20 p n {\left|\sg' \right|}.
\]
\end{proof}

\begin{lem} \label{boundMfromC2}
$\left| \M \right| \geq \frac{1}{20} p n^2 \left| C_2 \right|$.
\end{lem}

\begin{proof}
For every vertex $x \in C_2$, the number of edges in $\FF[\Pi]$ that contain $x$ 
is at most $\cmincross pn^2$, but by Proposition~\ref{prop:cross_degree},  w.h.p.\@ the crossing degree of $x$ in $\G$, $d_{\Pi}(x)$, is at least $pn^2/10$. Thus, there are at least $pn^2/20$ edges of $\M$ incident to $x$,
so $\left| \M \right| \geq \left| C_2 \right| pn^2/20$.
\end{proof}

\begin{proof}[Proof of Lemma~\ref{lem:f5}] Let $\delta$ be as defined in \eqref{constants} and $K$ sufficiently large that all the previous lemmas and propositions are applicable. 
We now have three different lower bounds on the size of {$\M$}.  We will show that $\left| \M \right| > 3\left| \B_1 \right|$ by proving that no matter how the
edges of $\B_1$ are arranged, one of the above lower bounds on $\M$ is larger than $3\left| \B_1 \right|$.
To do this, we divide the edges of $\B_1$ into three classes. Let $D = A_1 \setminus C$.
\vspace{-1ex}
\begin{enumerate}[I.]
  \setlength{\itemsep}{1pt}
  \setlength{\parskip}{0pt}
  \setlength{\parsep}{0pt}
\item $\B_1(1) = \left\{ e \in \B_1 : \left| e \cap C \right| \geq 2 \textnormal{ or } 
                                           \left| e \cap D \right| \geq 2 \right\}$.
\item $\B_1(2) = \left\{ e \in \B_1 \setminus \B_1(1) : \left| e \cap C_1 \right| = 1 \right\}$. Note that every edge in $\B_1(2)$ contains
      a vertex in $C_1$, one in $D$ and one outside of $A_1$.
\item $\B_1(3) = {\B}_1 \setminus \B_1(1) \setminus \B_1(2)$.  Note that every edge in $\B_1(3)$ contains
      a vertex in $C_2$, one in $D$ and one outside of $A_1$.
\end{enumerate}
\vspace{-1ex}
We now consider the following three cases on $ |\B_1(i)|$. 


\begin{case} $3 |\B_1(1)| \ge |\B_1|.$

Let {$\sg'$}$= \sg[C] \cup \sg[D]$.  By definition, vertices $x \in D$ have degree at most $\cdefineC n$.
For $x \in C$, Lemma~\ref{boundonC} shows that $x$ has degree in {$\sg'$} at most $\left| C \right| \leq
\cdeltal n < \cdefineC n$.
Proposition~\ref{prop:codegree} shows that $\left| \B_1(1) \right| \leq 2pn |\sg'|.$
Combined with Lemma~\ref{subgraphwithlargemindeg}, this shows that
$\left| \M \right| \geq 
20 pn |\sg'| \geq 10 \left| \B_1(1) \right| > 3 |\B_1|$.
\end{case}

\begin{case} $3 |\B_1(2)| \ge |\B_1|.$

For each vertex $x \in C_1$ and each $y \in D$, by Proposition~\ref{prop:codegree}, 
the co-degree of $x$ and $y$ is at most $2 pn$.  Since $\left| D \right| \leq n$,
there are at most $2pn^2$ edges of $\B_1 \setminus \B_1(1)$ containing $x$.
Thus $\left| \B_1(2) \right| \leq 2 p n^2 \left| C_1 \right| $, so Lemma~\ref{boundMfromC1} implies that
$\left| \M \right| \geq 
20 pn^2 \left| C_1 \right| \geq 10 \left| \B_1(2) \right|  > 3\left| \B_1 \right|$.
\end{case}

\begin{case} $3 |\B_1(3)| \ge |\B_1|.$

Every $x \in C_2$ is in less than $\cmincross p n^2$ crossing edges of $\FF$ .
Note that every edge in $\B_1(3)$ has at least one vertex in $C_2$ and is not completely contained in $A_1$ 
(edges completely contained in $A_1$ are in $\B_1(1)$.)
If there exist at least $\cmincross p n^2$ edges of {$\B_1$} which contain $x$
and have a vertex in $A_2$, we could move $x$ to $A_3$ and increase the number of edges
across the partition.  Similarly, there are at most $\cmincross pn^2$ edges of {$\B_1$} which contain $x$
and have a vertex in $A_3$, since otherwise we could move $x$ to $A_2$.
Thus $\left| \B_1(3) \right| \leq 2 \cmincross p n^2 \left| C_2 \right| = \frac{1}{200} p n^2 \left| C_2 \right|$.
Then Lemma~\ref{boundMfromC2} implies that 
$\left| \M \right| \geq \frac{1}{20} p n^2 \left| C_2 \right| \geq 10 \left| \B_1(3) \right| > 3\left| \B_1 \right|$.
\end{case}
{In each case we verified $|\M|>3|\B_1|$, then} since one of these three cases must hold, we have $|\M| > 3|\B_1|$.
\end{proof}

\subsection{Proof of Lemma~\ref{lem:key}}
\begin{proof}
{If $Q(\Pi)=\emptyset$, then clearly Lemma~\ref{lem:key} is true, so we may assume $Q(\Pi)\ne\emptyset$.}

Let $$\eps = 0.1,\  \zeta = 0.001,\  \gamma = 0.1,\  \alp = \frac{8}{9}\ \textnormal{and}\ \ph = 0.001.$$

{Recall that a partition $\Pi = (A_1, A_2, A_3)$ is balanced if $|A_i| = (1\pm10^{-10})n/3$ for every $i$, and for a balanced partition $\Pi$, $Q(\Pi) = \{(u,v)\in \binom{A_1}{2}: d_{\Pi}(u,v)<0.8 n^2p^2/9\}$.}


{  By Propositions~\ref{prop:common_degree} and \ref{prop:cross_degree}, for any balanced partition $\Pi$ we have \whp $d_{\Pi}(v)\ge (1-\eps)pn^2/9 = \gamma pn^2$ for every vertex $v$ and $d(u, v)\le (1+\eps)n^2p^2/2$ for every pair $(u,v)$ of distinct vertices, and therefore  $d_{\Pi}(u,v) \le \alp pd_{\Pi}(v)$ for every  $(u,v) \in Q(\Pi)$.
}

 Fix any positive $\delta<\ph\gamma/2$, and let $K$ be sufficiently large that all previous lemmas and propositions are applicable. Let $A = A(\delta)$ be the event that for $\delta >0$, there exists a balanced {partition} $\Pi$ such that $\t(\G) \le |\G(\Pi)| + |Q(\Pi)|\delta n^2p^2$. To prove Lemma~\ref{lem:key}, we will show that $\P[A] = o(1)$.
Since $Q(\Pi)$ contains a bipartite subgraph $R$ with at least half of the edges of $Q(\Pi)$, the event $A$ implies that $\t(\G) \le |\G(\Pi)| + 2|R|\delta n^2p^2$ for some bipartite $R\subseteq Q(\Pi)$. 
By Proposition~\ref{prop:bad_degree}, we have $d_{Q(\Pi)}(v)\le \zeta/p$ for every vertex $v$, and therefore,  we have
\beq{eq:bd}
d_R(v) \le  \zeta/p.
\eeq

Let $X, Y$ be disjoint subsets of $V$, $R$ be a spanning subgraph of $[X, Y]$ satisfying \eqref{eq:bd}, and $f$ be a function from $X$ to $\left\{ k\in \N: k \ge  \gamma p n^2 \right\}$. Denote by $E(R,X,Y,f)$ the event that there is a balanced {partition} $\Sigma$ of $\G$ such that for every vertex $x$ in $X$, we have
\beq{eq:cd}d_{\Sigma}(x) = f(x), \qquad R\subseteq Q(\Sigma) \qquad \textnormal{and}\qquad   \t(\G) \le |\G[\Sigma]| + \ph |R|\gamma n^2p^2, \eeq 
where we should emphasize that $Q(\Sigma)$ should be in the first partition class of $\Sigma$.
{Since} $\delta< \ph\gamma/2$,  the event $A$ implies event $E(R,X,Y,f)$ for some choice of $(R,X,Y,f)$.

We will show that there exists a constant $c$ such that
\beq{eq:keyP}
\P[E(R,X,Y,f)] \le e^{-c|R|n^2p^2}.
\eeq

There are at most $\binom{\binom{n}{2}}{t} 2^tn^{2t}$ ways to choose $(R,X,Y,f)$ with $|R|=t$. Then by the union bound, we have
\[
\P[A]\le \sum_{t\ge 1} \binom{\binom{n}{2}}{t} 2^tn^{2t} e^{-ctn^2p^2} \le    \sum_{t\ge 1} \left( \frac{en^4 }{t\cdot e^{cn^2p^2}}\right)^t= o(1).
\]

Now we prove \eqref{eq:keyP}, which completes the proof of Lemma~\ref{lem:key}.
We consider revealing the edges of $\G$ in stages:
\vspace{-1ex}
\begin{enumerate}[(i)]
  \setlength{\itemsep}{1pt}
  \setlength{\parskip}{0pt}
  \setlength{\parsep}{0pt}
  \item {Reveal}  the triplets of vertices of $\G$ that contain $x\in X$.
  \item {Reveal} the rest of the triplets of vertices of $\G$ except those belonging to $\bigcup_{y\in Y}[y, \cup_{xy\in R} L(x)]$.
  \item {Reveal} the rest of the triplets of vertices of $\G$.
\end{enumerate}
\vspace{-1ex}
Let $\G'$ be the subhypergraph of $\G$ consisting of the edges chosen in (i) and (ii), and let $\Gamma$ be a balanced {partition} of $\G'$ maximizing $|\G'[\Sigma]|$ among balanced {partitions} $\Sigma$ satisfying \eqref{eq:cd}.
Recall that for any balanced  {partition} $\Sigma$, we have $d_{\Sigma}(x,y)<\alp pd_{\Sigma}(x)$  for all $\left( x,y \right)\in Q(\Sigma)$. So for any balanced {partition} $\Sigma$ satisfying \eqref{eq:cd}, we have
\begin{eqnarray}
  |\G[\Sigma]| &\le& |\G'[\Sigma]| + \sum_{y\in Y}\sum_{xy\in R}d_{\Sigma}(x,y)
  \le |\G'[\Gamma]| + \sum_{y\in Y}\sum_{xy\in R}d_{\Sigma}(x,y)\nonumber\\
 & \le& |\G'[\Gamma]| + \sum_{y\in Y}\sum_{xy\in R} \alp p d_{\Sigma}(x)
\le |\G'[\Gamma]| + \alp p\sum_{y\in Y}\sum_{xy\in R} f(x).
  \label{eq:sigma}
\end{eqnarray}

Note that the right hand side of \eqref{eq:sigma} does not depend on the partition $\Sigma$, so it gives an upper bound on $|\G[\Sigma]|$ for all $\Sigma$ satisfying \eqref{eq:cd}. On the other hand, we look at $\Gamma$. For each $y\in Y$, set $M(y) = \cup_{xy\in R} L_{\Gamma}(x)$. We have
\beq{eq:tri}
\t(\G) \ge |\G[\Gamma]| =  |\G'[\Gamma]| + \sum_{y\in Y} |\G[y,M(y)]|.
\eeq

Recall that $d_{\Gamma}(x)  = f(x) \ge \gamma pn^2$, so for any two vertices $x$ and $x'$, we have $d_{\Gamma}(x,x')\le d(x,x')\le (1+\eps)n^2p^2/2 \le pd_{\Gamma}(x)/\gamma$. Also
recall that $R$ satisfies \eqref{eq:bd},
so for each $y\in Y$ we have $d_R(y)\le \zeta/p$. It follows that for each $y\in Y$, we have
\[
	\begin{array}{rclcl}
  	|M(y)|	& \ge& 	\displaystyle\sum_{xy\in R}\left[ d_{\Gamma}(x)-\sum_{x\ne x'\in N_R(y)} d_{\Gamma}(x, x') \right] 
			& \ge&	\displaystyle\sum_{xy\in R}\left[ d_{\Gamma}(x) - d_R(y)\cdot \displaystyle\max_{x\ne x'\in N_R(y)} d_{\Gamma}(x, x')\right] \\
			& \ge&	\displaystyle\sum_{xy\in R}\left[ d_{\Gamma}(x) - \zeta/p \cdot pd_{\Gamma}(x)/\gamma \right]  
			& \ge&	\displaystyle(1-\zeta/\gamma)\sum_{xy\in R}f(x).
	\end{array}
\]
Let $\mu$ be the expectation of the sum in \eqref{eq:tri}. Then we have 
$$\mu =p\sum_{y\in Y}|M(y)| \ge (1-\zeta/\gamma)p\sum_{y\in Y}\sum_{xy\in R}f(x). $$
 Then using Lemma~\ref{chernoff}, we know that with probability at least $1-e^{-c_{\eps}\mu} \ge 1- e^{-c|R|n^2p^2}$ for constant $c = c_{\eps}(\gamma - \zeta)$, the sum in \eqref{eq:tri} is at least $(1-\eps)\mu$, and when this happens, \eqref{eq:sigma} and \eqref{eq:tri} imply that
\[
t(\G) - |\G[\Sigma]| \ge \left((1-\eps)(1 - \zeta/\gamma) - \alp\right)p\sum_{y\in Y}\sum_{xy\in R}f(x) > \ph |R|\gamma n^2p^2,
\]
which proves \eqref{eq:keyP}.
\end{proof}

\section{Proof of Theorem~\ref{thm:main}}\label{Proof}

\begin{proof}[Proof of Theorem~\ref{thm:main}]
Let $\tilde{\FF}$ be a maximum $F_5$-free subhypergraph of $\G$, so $|\tilde{\FF}|\ge \t(\G)$.
Suppose to the contrary that $\tilde{\FF}$ is not tripartite; then
{to} prove Theorem~\ref{thm:main}, it suffices to show that $|\tilde{\FF}| < \t(\G)$. Let $\Pi = (A_1, A_2, A_3)$ be a $3$-partition maximizing $\tilde{\FF}[\Pi]$.  By  Proposition~\ref{prop:f5free}  we know that  $\Pi$ is balanced  and $\tilde{\FF}$ and $\Pi$ satisfy Condition~\eqref{lem:f5:1} of Lemma~\ref{lem:f5}. For $1 \leq i \leq 3$, let $\tilde{\B}_i = \{e\in \tilde{\FF}, |e\cap A_i| \ge 2\}$.
Without loss of generality, we may assume $|\tilde{\B}_1|\ge |\tilde{\B}_2|, |\tilde{\B}_3|$. 
Let $\B(\Pi)  = \{e\in \G : \exists (u,v)\in Q(\Pi)\textnormal{ s.t.\@ } \{u,v\}\subset e\}$ and $\FF = \tilde{\FF} - \B(\Pi)$. Observe that $\Pi$ is a maximal partition of 
$\tilde{\FF}$, and $\FF$ was obtained by removing some non-crossing edges of $\tilde{\FF}$, therefore $\Pi$ is a maximal partition of $\FF$ as well.
Now $\FF$ {and $\Pi$} satisfy all conditions of Lemma~\ref{lem:f5}. 
For $1 \leq i \leq 3$, let $\B_i = \{e\in \FF:\  |e\cap A_i| \ge 2\}$.
Then we have:
\begin{eqnarray}
|\tilde{\FF}|&\le& |\tilde{\FF}[\Pi]| + 3|\tilde{\B}_1|\nonumber\\
&=&|\FF[\Pi]| + 3|\B_1| + 3|\tilde{\FF}\cap \B(\Pi)|\nonumber\\
& \le & |\G[\Pi]|  + 3|\B(\Pi) | \label{eq:m:f5}\\
&\le &|\G[\Pi]| + 3\cdot 2 |Q(\Pi)|np\label{eq:m:cod}\\
&\le&|\G[\Pi]| + |Q(\Pi)|\delta n^2p^2\nonumber\\
&\le& \t(\G).\label{eq:m:key}
\end{eqnarray}

Here we apply Lemma~\ref{lem:f5} to $\FF$ and $\Pi$ to get \eqref{eq:m:f5}; note that equality is only possible when $\FF$ is tripartite. We apply   Proposition~\ref{prop:codegree} to get \eqref{eq:m:cod} and  Lemma~\ref{lem:key} to get \eqref{eq:m:key}. If $\FF$ is tripartite, but $\tilde\FF$ is not, then $Q(\Pi)\ne \emptyset$, and so equality in \eqref{eq:m:key} would fail.
{We therefore know that $|\tilde{\FF}| < \t(\G)$, a contradiction, which means  $\tilde{\FF}$ is {\tpt}.}
\end{proof}

\section*{Acknowledgement}

The authors would like to thank the anonymous referees for  very careful reading of the manuscript, and for several helpful suggestions.


\end{document}